\documentclass{amsart}
\usepackage{mathtools, microtype, mathrsfs, xfrac, hyperref, amssymb, enumerate, xspace, stmaryrd}
\usepackage[utf8]{inputenc}
\usepackage{tikz-cd}
\usepackage{pdflscape}
\usepackage{bm}

\usepackage{todonotes}

\numberwithin{equation}{section}

\numberwithin{subsection}{section}

\allowdisplaybreaks[1]

\newtheorem*{namedtheorem}{\theoremname}
\newcommand{\theoremname}{testing}

\newtheorem{theorem}{Theorem}
\newtheorem{proposition}[theorem]{Proposition}
\newtheorem{proposition-definition}[theorem]
{Proposition-Definition}
\newtheorem{corollary}[theorem]{Corollary}
\newtheorem{lemma}[theorem]{Lemma}
\newtheorem*{theorem*}{Theorem}

\theoremstyle{definition}
\newtheorem{definition}[theorem]{Definition}

\newtheorem{example}[theorem]{Example}

\newtheorem*{question*}{Question}

\theoremstyle{remark}

\renewcommand{\mathcal}{\mathscr}

\newcommand\cA{\mathcal{A}} \newcommand\cB{\mathcal{B}}
\newcommand\cC{\mathcal{C}} 
 
\newcommand\cG{\mathcal{G}}

\newcommand\cM{\mathcal{M}} \newcommand\cN{\mathcal{N}}
\newcommand\cO{\mathcal{O}}

\newcommand\cU{\mathcal{U}} 
 \newcommand\cX{\mathcal{X}}
 \newcommand\cZ{\mathcal{Z}}

 \newcommand\PP{\mathbb{P}}
\newcommand\QQ{\mathbb{Q}} \newcommand\RR{\mathbb{R}}

 \newcommand\bX{\mathbf{X}}

\newcommand\rK{\mathrm{K}}

\newcommand\fF{\mathfrak{F}} 
\newcommand\fG{\mathfrak{G}}

\newcommand\fN{\mathfrak{N}}

\newcommand\fQ{\mathfrak{Q}}

\newcommand\fX{\mathfrak{X}}

\newcommand\arr{\ifinner\to\else\longrightarrow\fi}

\newcommand\arrto{\ifinner\mapsto\else\longmapsto\fi}

\renewcommand\H{\operatorname{H}}

\newcommand{\eqdef}{\mathrel{\smash{\overset{\mathrm{\scriptscriptstyle def}} =}}}

\newcommand\im{\operatorname{im}}

\def\displaytimes_#1{\mathrel{\mathop{\times}\limits_{#1}}}

\def\displayotimes_#1{\mathrel{\mathop{\bigotimes}\limits_{#1}}}

\newcommand\aut{\operatorname{Aut}}

\newcommand\spec{\operatorname{Spec}}

\newcommand{\underaut}{\mathop{\underline{\mathrm{Aut}}}\nolimits}

\newlength{\ignora}

\renewcommand{\setminus}{\smallsetminus}

\newcommand{\PGL}{\mathrm{PGL}}

\newcommand{\gal}{\operatorname{Gal}}

\DeclareFontFamily{U}{mathx}{\hyphenchar\font45}
\DeclareFontShape{U}{mathx}{m}{n}{
      <5> <6> <7> <8> <9> <10>
      <10.95> <12> <14.4> <17.28> <20.74> <24.88>
      mathx10
      }{}
\DeclareSymbolFont{mathx}{U}{mathx}{m}{n}
\DeclareFontSubstitution{U}{mathx}{m}{n}
\DeclareMathAccent{\widecheck}{0}{mathx}{"71}
\DeclareMathAccent{\wideparen}{0}{mathx}{"75}

\renewcommand{\epsilon}{\varepsilon}

\newcommand{\cha}{\operatorname{char}}

\newcommand{\aff}[1][k]{(\mathrm{Aff}/#1)}
\newcommand{\as}[1][k]{(\mathrm{FAS}/#1)}

\begin{document}

\title{The field of moduli of varieties with a structure}

\author{Giulio Bresciani}

\begin{abstract}
	If $X$ is a variety with an additional structure $\xi$, such as a marked point, a divisor, a polarization, a group structure and so forth, then it is possible to study whether the pair $(X,\xi)$ is defined over the field of moduli. There exists a precise definition of ``algebraic structures'' which covers essentially all of the obvious concrete examples. We prove several formal results about algebraic structures. There are immediate applications to the study of fields of moduli of curves and finite sets in $\mathbb{P}^{2}$, but the results are completely general.
	
	Fix $G$ a finite group of automorphisms of $X$, a $G$-structure is an algebraic structure with automorphism group equal to $G$. First, we prove that $G$-structures on $X$ are in a $1:1$ correspondence with twisted forms of $X/G\dashrightarrow\cB G$. Secondly we show that, under some assumptions, every algebraic structure on $X$ is equivalent to the structure given by some $0$-cycle. Third, we give a cohomological criterion for checking the existence of $G$-structures not defined over the field of moduli. Fourth, we identify geometric conditions about the action of $G$ on $X$ which ensure that every $G$-structure is defined over the field of moduli.
\end{abstract}

\address{Scuola Normale Superiore\\Piazza dei Cavalieri 7\\
56126 Pisa\\ Italy}
\email{giulio.bresciani@gmail.com}

\maketitle
\section{Introduction}

Let $K/k$ be a possibly infinite Galois extension of fields. Consider a variety $X$ over $K$ with some additional structure $\xi$: for instance, $\xi$ might be a preferred point $x\in X(K)$, or a polarization, or a group scheme structure (see \cite[\S 5]{giulio-angelo-moduli} for a precise definition of ``structures'').

Given a Galois element $\sigma\in\gal(K/k)$, consider the twist $\sigma^{*}(X,\xi)$: if $X$ is defined by polynomials, this corresponds to applying $\sigma$ to the coefficients of the polynomials. It can be shown that the subgroup $H\subset\gal(K/k)$ of elements $\sigma$ such that $\sigma^{*}(X,\xi)\simeq (X,\xi)$ is open, and the field of elements of $K$ fixed by $H$ is called \emph{the field of moduli} of $(X,\xi)$. If $(X,\xi)$ descends to some subextension $K/k'/k$, then $k'$ contains the field of moduli. The following is the basic question.

\begin{question*}
	Is $(X,\xi)$ defined over its field of moduli?
\end{question*} 

One of the oldest known results regarding this question is the fact that an elliptic curve $E$ over $\bar{\QQ}$ is defined over $\QQ(j_{E})$ where $j_{E}$ is the $j$-invariant of $E$; this result predates the concept of field of moduli by several decades. Fields of moduli were introduced by Matsusaka \cite{matsusaka-field-of-moduli} in 1958 and later clarified by Shimura \cite{shimura-automorphic}, who also proved that a generic, principally polarized abelian variety of odd dimension in characteristic $0$ is defined over its field of moduli \cite{shimura}. They have been studied intensively over the years, mainly for curves and abelian varieties, see for instance \cite{koizumi} \cite{murabayashi} \cite{debes-douai} \cite{debes-emsalem} \cite{huggins} \cite{cardona-quer} \cite{kontogeorgis} \cite{marinatto}.

One main reason why results are restricted to curves and abelian varieties is the lack of appropriate technology. In particular, a lot of results about curves rely on results by Dèbes, Douai and Emsalem contained in \cite{debes-douai} \cite{debes-emsalem} which until recently were only available in dimension $1$. In our joint article with A. Vistoli \cite{giulio-angelo-moduli}, we have generalized and clarified their methods and results: there are now a general framework and general techniques for studying fields of moduli of varieties with a structure in arbitrary dimension. As an example of our new techniques, we reproved and generalized Shimura's result about abelian varieties in a much more theoretical fashion \cite[Corollary 6.25]{giulio-angelo-moduli}.

The study of fields of moduli of higher dimensional varieties is thus a largely unexplored topic. As a first application of our joint work with Vistoli \cite{giulio-angelo-moduli} to an open problem, in \cite{giulio-p2} \cite{giulio-points} \cite{giulio-realcomplex} we study the field of moduli of curves and finite subsets of $\PP^{2}$. Among other things, we prove that every smooth plane curve of degree prime with $6$ is defined over the field of moduli, and that odd degree is sufficient if the base field is $\RR$.

This brief note is an expansion of the general technology constructed in \cite{giulio-angelo-moduli}. The results contained here are crucial for our work on $\PP^{2}$, but completely general in nature, and will be applied in more forthcoming works about fields of moduli of higher dimensional varieties. 

\subsection*{Acknowledgements}

I would like to thank an anonymous referee for providing useful remarks.

\subsection*{Notation}

We use $\underaut$ for sheaves of groups of automorphisms, and $\aut$ for set-theoretic groups of automorphisms.  

\section{Contents of the paper}

For the rest of the paper, $k$ is a field with separable closure $K$, and $X$ is a separated, connected, normal algebraic space of finite type over $K$.

In order to study the problem of the field of moduli of the space $X$ with some additional structure $\xi$, such as a group structure, marked points, effective divisors, polarizations and so forth we need to define precisely what ``an algebraic structure'' is. We have given such a formalization using stacks in \cite[\S 5]{giulio-angelo-moduli}, and we recall it here in \S\ref{sect:str}.

In our intentions, the concept of algebraic structure was nothing more than a unifying definition. It turns out that studying algebraic structures as independent, abstract objects not necessarily tied to some geometric meaning leads downstream to insights about the original problem of fields of moduli of actual, geometric objects. Let us give an example.

In \cite{giulio-p2} we prove that, for the large majority of finite subgroups $G\subset\PGL_{3}(\bar{\QQ})$, \emph{every} algebraic structure on $\PP^{2}_{\bar{\QQ}}$ with automorphism group equal to $G$ is defined over its field of moduli. This is completely independent of the geometric origin of the structure: the result holds for sets of points, for smooth curves, for cycles and so forth. The general techniques which we give here have nothing to do with $\PP^{2}$, though, so the same kind of analysis can be done for other varieties.

Fix $G\subset\aut_{K}(X)$ is a finite group of automorphisms of $X$ of order prime with $\cha k$. A $G$-structure is an algebraic structure $\xi$ on $X$ such that the group scheme of automorphism $\underaut(X,\xi)$ is constant and equal to $G$.

\subsection{Equivalent structures}

Given a $G$-structure $\xi$ on $X$, the concept of field of moduli $k_{\xi}$ of $(X,\xi)$ is defined, and there is a finite étale gerbe $\cG_{\xi}$ over the field of moduli $k_{\xi}$ called the \emph{residual gerbe} \cite[\S 3.1]{giulio-angelo-moduli} which parametrizes twisted forms of $(X,\xi)$, and a universal family $\cX_{\xi}\to\cG_{\xi}$ whose fibers are the corresponding twists of $X$ \cite[\S 5]{giulio-angelo-moduli}. In particular, $\xi$ is defined over $k_{\xi}$ if and only if $\cG_{\xi}(k_{\xi})\neq\emptyset$. The coarse moduli space $\bX_{\xi}$ of $\cX_{\xi}$ is called the \emph{compression} of $\xi$, and there is an induced rational map $\bX_{\xi}\dashrightarrow\cG_{\xi}$.

We regard two $G$-structures as equivalent if they have the same field of moduli and isomorphic universal families over isomorphic residual gerbes. Loosely speaking, two $G$-structures are equivalent if they contain essentially the same data and hence their respective twisted forms are in a $1:1$ correspondence.

\subsection{Twisted $G$-quotients}

Our first result is a classification of algebraic structures up to equivalence.

\begin{definition}
    A \emph{twisted $G$-quotient} of $X$ over $k$ is an algebraic space $Y$ over $k$ with an isomorphism $Y_{K}\simeq X/G$ such that for every $\sigma\in\gal(K/k)$ there exists a $\sigma$-linear automorphism of $X$ which commutes with the composition $X\to X/G\simeq Y_{K}\to Y$. Equivalently (see Corollary~\ref{corollary:twisteq}), a twisted $G$-quotient is a rational map $Y\dashrightarrow \cG$ which is a twisted form over $k$ of $X/G\dashrightarrow \cB G$.
\end{definition}

We say that two twisted $G$-quotients are equivalent if there is an isomorphism which respects the identification with $X/G$.

\begin{theorem}\label{theorem:GstrGtwist}
	Mapping an algebraic structure $\xi$ with field of moduli $k$ to the compression $\bX_{\xi}\dashrightarrow\cG_{\xi}$ defines a one-to-one correspondence between $G$-structures on $X$ and twisted $G$-quotients of $X$ over $k$, up to equivalence.
	
	The structure $\xi$ is defined over the field of moduli if and only if the projection $X\to X/G$ descends to a ramified covering of the compression $\bX_{\xi}$.
\end{theorem}

Thanks to Theorem~\ref{theorem:GstrGtwist}, one can study twisted $G$-quotients directly, forgetting about the original structure. If one can prove that, for a given $G$, every twisted $G$-quotient $Y$ has a smooth (more generally, liftable cf. \cite[Definition 6.6]{giulio-angelo-moduli}) rational point, then we automatically get that every $G$-structure $\xi$ is defined over its field of moduli by the Lang-Nishimura theorem for tame stacks \cite[Theorem 4.1]{giulio-angelo-valuative} applied to $\bX_{\xi}\dashrightarrow\cG_{\xi}$, regardless of whether $\xi$ was the datum of a point, a divisor, a group structure or anything else. In fact, this is the case for most finite subgroups of both $\underaut(\PP^{1})=\PGL_{2}$ \cite{giulio-divisor} and $\underaut(\PP^{2})=\PGL_{3}$ \cite{giulio-p2}.

\subsection{Interpretation as cycle-structures} If, on the other hand, we find a twisted $G$-quotient $Y$ over $k$ whose associated structure is not defined over $k$, we can search for a meaningful structure whose compression is $Y$. This is the technique we used to construct examples of cycles on $\PP^{1}$ \cite{giulio-divisor} and on $\PP^{2}$ \cite{giulio-p2} \cite{giulio-points} not defined over their field of moduli. The next result says that, under some assumptions, it is always possible to give such an interpretation using $0$-cycles.

\begin{theorem}\label{theorem:cycles}
	Assume that $\cha k=0$, that $\underaut(X)$ is of finite type over $k$, and that there exists some finite extension $k'/k$ with a model $\fX'$ over $k'$ such that $\fX'(k')$ is dense.
	
	For every algebraic structure $\xi$ on $X$, there exists a $0$-cycle $Z$ on $X$ such that $(X,Z)$ is equivalent to $(X,\xi)$. 
\end{theorem}

\subsection{Cohomological criterion}

Using non-abelian cohomology, we give a criterion to study the existence of $G$-structures not defined over their field of moduli. Suppose that $K$ is separably closed, let $\fX$ be an integral algebraic space of finite type over $k$ with $\fX_{K}=X$, and $\fG$ a group scheme finite étale over $k$ with a faithful action on $\fX$ such that $\fG_{K}=G\subset\underaut(X)$. Let $\fN\subset \underaut(\fX)$ a subgroup sheaf which normalizes $\fG$, and $\fQ=\fN/\fG$ the quotient.

\begin{theorem}\label{theorem:cohomology}
	If the natural map $\H^{1}(k,\fN)\to\H^{1}(k,\fQ)$ is not surjective, then there exists a $G$-structure on $X$ with field of moduli $k$ which does not descend to $k$. If $\fN$ is the entire normalizer of $\fG$, the converse holds.
\end{theorem}

\subsection{Rational points on twisted quotients}

While Theorem~\ref{theorem:cohomology} is mostly useful to construct counterexamples by exhibiting an actual element of $\H^{1}(k,\fQ)$ which does not lift to $\fN$, it is harder to apply it in the other direction, namely showing that $\H^{1}(k,\fN)\to\H^{1}(k,\fQ)$ is surjective when $\fN$ is the entire normalizer of $\fG$. Usually, there is a more fruitful strategy for trying to show that, for fixed $G$, an arbitrary $G$-structure is defined over the field of moduli. Assume that $|G|$ is prime with $\cha k$.

Let $Y\dashrightarrow\cG$ be a twisted $G$-quotient over $k$. If we can find a smooth rational point $y\in Y(k)$, since $|G|$ is prime with $\cha p$ then by the Lang-Nishimura theorem for tame stacks \cite[Theorem 4.1]{giulio-angelo-valuative} we get that $\cG(k)\neq\emptyset$, hence the associated structure is defined over the field of moduli. Furthermore, the smoothness assumption can be relaxed, see \cite[\S 6]{giulio-angelo-moduli} \cite{giulio-tqs2}. It is then important to clarify under which conditions a closed subset $Z\subset X$ descends to a rational point (or more generally to a subspace) of every twisted $G$-quotient. We identify and study such conditions in \S \ref{sec:dist}.

\section{Algebraic structures}\label{sect:str}

Let us recall how we formalized the concept of algebraic structures in \cite{giulio-angelo-moduli}.

Denote by $\as$ the fibered category over $\aff$ whose objects over an affine scheme $S$ over $k$ are flat, finitely presented morphisms $Y\to S$, where $Y$ is an algebraic space. A \emph{category of structured spaces} over $k$ is a locally finitely presented fppf stack $\cM\to\aff$ with a faithful cartesian functor $\cM\to\as$.

The following definition is implicit in \cite{giulio-angelo-moduli}.

\begin{definition}
	An \emph{algebraic structure} over $X$ is the datum of a category of structured spaces $\cM\to\aff$ with an algebraic (cf. \cite[Definition 3.6]{giulio-angelo-moduli}) morphism $\spec K\to \cM$ whose composition with $\cM\to\as$ is the $K$-point of $\as$ corresponding to $X$.
\end{definition}

Given a structure $\xi$ on $X$, the field of moduli $k_{\xi}$ of $(X,\xi)$ is well defined \cite[Definition 3.11]{giulio-angelo-moduli}, and there is an algebraic stack $\cG_{\xi}$ together with a universal family $\cX_{\xi}\to\cG_{\xi}$ which parametrizes twisted forms of $(X,\xi)$. The stack $\cG_{\xi}$ is a gerbe of finite type over $k_{\xi}$ \cite[\S 3.1]{giulio-angelo-moduli}, and is called the residual gerbe of $(X,\xi)$. The residual gerbe is naturally equipped with a fully faithful morphism $\cG_{\xi}\to \cM$.

As we have said in the introduction, we regard two algebraic structures on $X$ as equivalent if they have the same field of moduli and isomorphic universal families over isomorphic residual gerbes.

Since the residual gerbe together with the composition $\cG_{\xi}\to\cM\to\as$ is itself a category of structured spaces, we get the following fact, which can be regarded as a shortcut for the definition of algebraic structure. The proof is straightforward.

\begin{lemma}\label{lemma:str}
	Mapping an algebraic structure $\xi$ to its universal family $\cX_{\xi}\to\cG_{\xi}$ defines a bijection between
	\begin{itemize}
		\item algebraic structures on $X$ up to equivalence, and
		\item pair of morphisms $p:\spec K\to\cG$, $\cX\to\cG$ with an isomorphism $\cX\times_{\cG}\spec K\xrightarrow{\sim} X$ where $\cG$ is a finite gerbe over a finite subextension of $K/k$ such that the induced action of $\underaut_{\cG}(p)$ on $X$ is faithful, up to equivalence.
	\end{itemize}
\end{lemma}

\section{Proof of Theorem~\ref{theorem:GstrGtwist}}\label{sect:twist}

Consider the groups $\aut_{k}(X),\aut_{K}(X)$ of automorphisms of $X$ over $k$ and $K$ respectively. A $k$-automorphism of $X$ induces a $k$-automorphism of $K\subset\H^{0}(X,\cO)$, since $K$ is the separable closure of $k$ in $\H^{0}(X,\cO)$. Hence, there is a short exact sequence
\[1\to\aut_{K}(X)\to\aut_{k}(X)\to\gal(K/k)\]
where the image of the right arrow is the Galois group of the field of moduli of $X$, see \cite[Proposition 3.13]{giulio-angelo-moduli}.

For ease of notation, from now on we always assume that twisted quotients are defined over $k$, and that structures have field of moduli equal to $k$: by base changing, it is always possible to reduce to this case.

Let $G\subset\aut_{K}(X)$ be a finite group and $Y$ a twisted $G$-quotient of $X$ over $k$ with natural projection $q:X\to X/G\simeq Y_{K}\to Y$.

\begin{definition}
	The \emph{distinctive subgroup} $\cN_{Y}\subset\aut_{k}(X)$ of $Y$ is the subgroup of automorphisms $\phi$ of $X$ over $k$ such that the diagram
	
	\[\begin{tikzcd}
	X\ar[dr,swap,"q"]\ar[rr,"\phi"]		&	&	X\ar[dl,"q"]	\\
			&		Y		&
	\end{tikzcd}\]
	commutes. If $\xi$ is an algebraic structure on $X$, the distinctive subgroup of $\xi$ is the distinctive subgroup of its compression, and we simply write $\cN_{\xi}$.
\end{definition}

The composition $\cN_{Y}\subset\aut_{k}(X)\to\gal(K/k)$ is surjective by definition of twisted quotient. Furthermore, since $X$ is separated and integral, any automorphism of $X$ which commutes with $X\to X/G$ is an element of $G$. We thus get a short exact sequence
\[1\to G\to\cN_{Y}\to\gal(K/k)\to 1.\]

\begin{proposition}\label{prop:Gtwist}
	Let $Y$ be a twisted $G$-quotient of $X$ over $k$. The quotient stack $[\spec K/\cN_{Y}]$ is a finite étale gerbe over $k$ and $[X/\cN_{Y}]\to[\spec K/\cN_{Y}]$ together with the tautological morphism $\spec K\to[\spec K/\cN_{Y}]$ defines an algebraic structure with field of moduli equal to $k$ and compression equal to $Y$.
\end{proposition}

\begin{proof}
	The morphism $[\spec K/\cN_{Y}]\to\spec k$ is induced by $\spec K\to\spec k$, which is by construction $\cN_{Y}$-invariant. Notice that $[\spec K/\cN_{Y}]_{K}\simeq [\spec K/G]=\cB G$, hence $[\spec K/\cN_{Y}]$ is a finite étale gerbe. Furthermore, we have
    \[\spec K\times_{[\spec K/\cN_{Y}]}[X/\cN_{Y}]\simeq \spec K\times_{[\spec K/G]}[X/G]\simeq X.\]
	We conclude by applying Lemma~\ref{lemma:str}.
\end{proof}

\begin{corollary}\label{corollary:eq}
	Two algebraic structures $\xi,\xi'$ on $X$ with étale automorphism groups are equivalent if and only if their distinctive subgroups are equal. \qed
\end{corollary}

\begin{corollary}\label{corollary:twisteq}
    The datum of a twisted $G$-quotient of $X$ over $k$ is equivalent to the datum of a twisted form $Y\dashrightarrow\cG$ of $X/G\dashrightarrow\cB G$ over $k$.
\end{corollary}

\begin{proof}
    If $Y$ is a twisted $G$-quotient over $k$, $X/\cN_{Y}\dashrightarrow [\spec K/\cN_{Y}]$ is a twisted form of $X/G\dashrightarrow \cB G$ over $k$. 
    
    Conversely, assume that $Y\dashrightarrow\cG$ is a twisted form of $X/G\dashrightarrow \cB G$ over $k$; let $\sigma\in\gal(K/k)$ be a Galois automorphism of $K/k$ and $V\subset Y$ be an open subset such that $V\to\cG$ is a morphism, write $U\subset X$ for the inverse image of $V$.
    
    Any two morphisms $\spec K\to\cG$ are isomorphic since $\cG$ is a finite étale gerbe and $K$ is separably closed. This implies that $V\times_{\cG} \spec K$ is isomorphic to both $U$ and $\sigma^{*}U$, hence we get an isomorphism $U\simeq\sigma^{*}U$ over $V$, or equivalently a $\sigma$-linear automorphism of $U$ over $V$. Since $X$ is normal, it coincides with the integral closure of $V$ in $U$; we thus get an induced $\sigma$-linear automorphism of $X$ over $Y$.
\end{proof}

The first part of Theorem~\ref{theorem:GstrGtwist} is a direct consequence of Lemma~\ref{lemma:str}, Proposition~\ref{prop:Gtwist} and Corollary~\ref{corollary:twisteq}.

If $Y\dashrightarrow\cG$ is a twisted form of $X/G\dashrightarrow\cB G$, a direct way of constructing the universal family $[X/\cN_{Y}]\to\cG=[\spec K/\cN_{Y}]$ without using distinctive subgroups is to consider the relative normalization of $Y\times \cG$ with respect to $\spec k(Y)\to Y\times\cG$.

Let us now look at the second part of Theorem~\ref{theorem:GstrGtwist}. If $\spec k\to\cG_{\xi}$ is a rational section, then $E\eqdef \spec k\times_{\cG_{\xi}}\cX_{\xi}$ with the composition $E\to\cX_{\xi}\to\bX_{\xi}$ defines a twisted form of $X\to X/G$ over $k$.

On the other hand, if $E\to\bX_{\xi}$ is a twisted form of $X\to X/G$, then $\aut(X/E)=\gal(K/k)\subset\cN_{\xi}$ defines a section of $\cN_{\xi}\to\gal(K/k)$, which in turn induces a morphism $\spec k\to\cG_{\xi}=[\spec K/\cN_{\xi}]$.

\section{Proof of Theorem~\ref{theorem:cycles}}

Let us construct the algebraic structure associated with a cycle $Z$ on $X$ with finite automorphism group scheme.

If $C\subset X$ is a reduced, irreducible closed subscheme and $S$ is a scheme over $k$, a twisted form of $C$ over $S$ is a flat, locally of finite type morphism $\cX\to S$ with a closed subscheme $\cC\subset \cX$ such that there exists a finite subextension $K/k'/k$, a scheme $S'$ over $k'$ and an fppf covering $S'\to S$ such that 
\[(\cX|_{S'}\times_{k'}K,\cC|_{S'}\times_{k'}K)\simeq (X\times_{K} S'_{K},C\times_{K} S'_{K}).\]
As expected, $(X,C)$ defines a twisted form of $(X,C)$: since $(X,C)$ descends to some finite subextension $K/k'/k$, then $(X\times_{k'}K,C\times_{k'}K)$ is isomorphic to $(X\times_{K}\spec K\otimes_{k'}K,C\times_{K}\spec K\otimes_{k'}K)$ and hence we may take $S'=S=\spec K$.

If $Z=\sum_{i}n_{i}C_{i}$ is a cycle, a twisted form of $Z$ over $S$ is a flat, locally of finite type morphism $\cX\to S$ with a formal sum $\sum_{i}n_{i}\cC_{i}$ where $\cC_{i}\subset\cX$ is a twist of $C_{i}$.

The residual gerbe $\cG_{Z}$ is the functor $S\mapsto$\{twisted forms of $Z$\}; if the group scheme of automorphisms $\underaut(X,Z)$ is finite then $\cG_{Z}$ is an algebraic stack which is a gerbe by \cite[Proposition 3.9]{giulio-angelo-moduli} and $\cX_{Z}\to\cG_{Z}$ is the corresponding universal family whose $S$-points are twisted forms $(\cX,\cZ)$ of $Z$ over $S$ with a section $S\to\cX$. The coarse moduli space of $\cG_{Z}$ is the spectrum of the field of moduli $k_{Z}$ of $Z$. 

\begin{lemma}
	If $\underaut(X,Z)=\aut_{K}(X,Z)$ is finite étale, then $\cN_{Z}=\aut_{k}(X,Z)\subset\aut_{k}(X)$.
\end{lemma}

\begin{proof}
	Let $\bX_{Z}$ be the compression, i.e. the coarse moduli space of $\cX_{Z}$, and denote by $c$ the morphism $X\to\cX_{Z}$. First, let us show that $\cX_{Z}\to\bX_{Z}$ is birational. Since $X$ is integral, there exists an open subset $U\subset X$ which is $\aut_{K}(X,Z)$-invariant and such that the action on $U$ is free; this induces an open subset $\cU\subset\cX_{Z}$. A morphism $s:S\to \cU$ corresponds to a twisted form $(\cX,\cZ)$ of $(X,Z)$ with a section $S\to\cX$ which lands in the subspace of $\cX$ corresponding to $U$. An automorphism of $s$ is an automorphism of the twisted form $(\cX,\cZ)$ which maps $s$ to itself. Since $s$ lands in the locus corresponding to $U$, and the action on $U$ is free, it is straightforward to check that such an automorphism is the identity. This implies that $\cU$ is an algebraic space, and thus $\cX_{Z}\to\bX_{Z}$ is birational. 
	
	Let us now prove the main statement. By definition, $\cN_{Z}$ is the group of $k$-automorphisms of $X$ which commute with the composition $X\to\cX_{Z}\to\bX_{Z}$. If $g$ is a $k$-linear automorphism of $(X,Z)$, then by definition of $\cX_{Z}$ the composition $c\circ g$ is $2$-isomorphic to $c$, so in particular $g\in\cN_{Z}$. On the other hand, if $g\in\cN_{Z}$, the two compositions of $c\circ g$ and $c$ with $\cX_{Z}\to\bX_{Z}$ are equal. Since $\cX_{Z}\to\bX_{Z}$ is birational and $c$ is surjective, we get that $c\circ g$ and $c$ are isomorphic on a dense open subset. Since $\cX_{Z}$ is separated (which follows from the fact that $X$ is separated) then $c\circ g$ and $c$ are isomorphic globally. By definition of $\cX_{Z}$, this implies that $g^{*}Z=Z$, hence $g\in\aut_{k}(X,Z)$.
\end{proof}	 

As a consequence, another way of constructing $\cX_{Z}\to\cG_{Z}$ is by considering the quotient stacks $[X/\aut_{k}(X,Z)]\to[\spec K/\aut_{k}(X,Z)]$ as in Proposition~\ref{prop:Gtwist}.

\begin{lemma}\label{lemma:avoid}
	Let $Y$ be a twisted $G$-quotient of $X$ over $k$ and $U\subset X$ a $\cN_{Y}$-invariant non-empty open subset, i.e. the inverse image of an open subset of $Y$. Assume that there exists a finite subextension $K/k'/k$ such that $Y(k')$ is dense. For every $\tau\in\aut_{k}(X)\setminus \cN_{Y}$, there exists a $\cN_{Y}$-invariant finite subset $Z\subset U(K)$ such that $\tau(Z)\neq Z$.
\end{lemma}

\begin{proof}
	Let $q:X\to X/G=Y_{K}\to Y$ be the composition, by hypothesis $q\circ\tau\neq\tau$. Up to enlarging $k'$, we may assume that $k'/k$ is Galois. We may find a finite \emph{subset} $H\subset\cN_{Y}$ such that $H\to\gal(k'/k)$ is surjective. Notice that we cannot impose that $H$ is a subgroup, but fortunately this is not necessary.
	
	If $x:\spec K\to U$ is a point such that $q\circ x:\spec K\to Y$ is $k'$-rational and $\phi\in\cN_{Y}$, $h\in H\subset\cN_{Y}$ are elements with equal images in $\gal(k'/k)$, then clearly
	\[q\circ\phi\circ x=q\circ h\circ x:\spec K\to Y.\]
	
	Since $\tau$ is not in $\cN_{Y}$, I claim that there exists a point $x:\spec K\to U$ whose image in $Y$ is $k'$-rational and such that $\tau \circ x\neq \phi\circ x$ for every $\phi\in\cN_{Y}$. If by contradiction this is false, then for every $x\in U(K)$ whose image in $Y$ is $k'$-rational we may choose $h_{x}\in H$ such that $q\circ \tau\circ x=q\circ h_{x}\circ x$. Since $H$ is finite and $k'$-rational points are dense, this implies that there exists an element $h\in H$ such that $q\circ \tau=q\circ h$. Since $h\in H\subset\cN_{Y}$, then $q\circ h=q$ and thus $q\circ\tau=q$, which is absurd.
	
	Let $x:\spec K\to U$ be a point whose image in $Y$ is $k'$-rational and such that $\tau \circ x\neq \phi\circ x$ for every $\phi\in\cN_{Y}$, and denote by $Z$ its $\cN_{Y}$-orbit. Clearly, $Z$ is $\cN_{Y}$-invariant and $\tau(Z)\neq Z$ since $\tau\circ x\in \tau(Z)$ and $\tau\circ x\not\in Z$.
\end{proof}

Let us now prove Theorem~\ref{theorem:cycles}. Thanks to Corollary~\ref{corollary:eq} and since we are in characteristic $0$, it is enough to find a $0$-cycle $Z$ on $X$ such that $\aut_{k}(X,Z)=\cN_{\xi}$.

Let $Z_{0}$ be the empty $0$-cycle, and define a $\cN_{\xi}$-invariant cycle $Z_{i}$ by recursion as follows. If $\aut_{k}(X,Z_{i})=\cN_{\xi}$, then $Z_{i+1}=Z_{i}$. Otherwise, choose $\tau\in \aut_{k}(X,Z_{i})\setminus \cN_{\xi}$ and define
\[Z_{i+1}=Z_{i}+(i+1)Z_{\tau},\]
where $Z_{\tau}\subset X(K)$ is a $\cN_{\xi}$-invariant subset such that $\tau(Z_{\tau})\neq Z_{\tau}$ and such that the support of $Z_{\tau}$ is disjoint from the support of $Z_{i}$, it exists by Lemma~\ref{lemma:avoid}. By construction, the coefficients of $Z_{i}$ are at most $i$ while the coefficients of $Z_{\tau}$ are equal to $i+1$. This implies that
\[\aut_{k}(X,Z_{i+1})=\aut_{k}(X,Z_{i})\cap\aut_{k}(X,Z_{\tau}).\]
In particular, $\cN_{\xi}\subset\aut_{k}(X,Z_{i+1})$ and $\tau\not\in\aut_{k}(X,Z_{i+1})$, hence $\aut_{k}(X,Z_{i+1})\subsetneq\aut_{k}(X,Z_{i})$ if $\aut_{k}(X,Z_{i})$ is different from $\cN_{\xi}$.

The group $\aut_{k}(X)$ is an extension of $\gal(K/k)$ by $\aut_{K}(X)=\underaut(X)(K)$. The former is a compact group with respect to the pro-finite topology, while the latter is compact with respect to the Zariski topology since $\underaut(X)$ is noetherian. Notice that the subgroups $\aut_{k}(X,Z_{i})\cap\aut_{K}(X)=\aut_{K}(X,Z_{i})$ and $\im(\aut_{k}(X,Z_{i})\to\gal(K/k))=\gal(K/k_{Z_{i}})$ are both closed. This implies that the chain $\aut_{k}(X,Z_{i})\supseteq\aut_{k}(X,Z_{i})$ is eventually stable, i.e. $\aut_{k}(X,Z_{i})=\cN_{\xi}$ for some $i>>0$. This concludes the proof of Theorem~\ref{theorem:cycles}.

\section{Proof of Theorem~\ref{theorem:cohomology}}

Let $\fG$ be a finite, étale group scheme over $k$ which acts faithfully on an algebraic space $\fX$ of finite type over $k$. Write $X=\fX_{K}$, $G=\fG_{K}$. We are going to show that twisted $G$-quotients of $X$ are obtained by twisting $\fX/\fG\dashrightarrow\cB_{k}\fG$ with some torsor.

Let $\underaut(\fX)$ be the sheaf of automorphisms of $\fX$. If $\fX$ is projective, this is representable, but we don't need this assumption. Let $\fN\subset\underaut(\fX)$ be a subgroup sheaf which normalizes $\fG$, and write $\fQ$ for the quotient $\fN/\fG$.

\begin{lemma}\label{lemma:autocover}
	If $\fN$ is the entire normalizer of $\fG$, the sheaf of groups $\fQ$ is isomorphic to the fppf sheaf $\fF$ of automorphisms of $\fX/\fG$ for which, fppf locally, there exists an automorphism of $\cB_{k}\fG$ making the obvious diagram $2$-commutative.
\end{lemma}

Notice that the particular automorphism of $\cB_{k}\fG$ is \emph{not} part of the definition of $\fF$, we only require existence.

\begin{proof}
	There is an obvious injective map $\fQ\to\fF$, let us show that it is surjective. If $S$ is a scheme over $k$ and $f$ is an automorphism of $\fX_{S}/\fG_{S}$ in $\fF(S)$, the hypothesis implies that there exists an fppf covering $S'\to S$ with an automorphism $h$ of $\fX_{S'}$ which lifts $f_{S'}$. The fact that $h$ lifts an automorphism of $\fX_{S'}/\fG_{S'}$ implies that $h\in\fN(S')$. While the class of $h$ is not well defined and depends on a choice, its image $q_{0}\in\fQ(S')$ is unique. The uniqueness of $q_{0}$ implies that it descends to an element $q\in\fQ(S)$. It is straightforward to check that the image of $q$ in $\fF$ is $f$.
\end{proof}

Let $T\to\spec k$ be a $\fQ$-torsor over $k$, it corresponds to a morphism $\spec k\to\cB_{k}\fQ$. Define a twisted form $\cX\to\cG$ of $[X/G]\to\cB_{K}G$ as the fibered products
\[\begin{tikzcd}
    \cX\rar\dar     &   \cG\rar\dar     &   \spec k\ar[d,"T"]    \\
    \left[\fX / \fN\right]\rar  &   \cB_{k}\fN\rar  &   \cB_{k}\fQ.
\end{tikzcd}\]
Since $T$ does not lift to $\fN$, then $\cG(k)=\emptyset$ and hence the corresponding structure is not defined over $k$.

Let us give another construction. There are actions of $\fQ$ on both $[\fX/\fG]$ and on $\cB_{k}\fG=[\spec k/\fG]$, and using Romagny's theory of group actions of stacks \cite{romagny} it is possible to define twists
\[[\fX/\fG]\times^{\fQ}T=[([\fX/\fG]\times T)/\fQ],~\cB_{k}\fG\times^{\fQ}T=[(\cB_{k}\fG\times T)/\fQ],\]
with an induced morphism
\[[\fX/\fG]\times^{\fQ}T\to \cB_{k}\fG\times^{\fQ}T\]
whose coarse moduli space is a twisted $G$-quotient. 

By definition, $\cB_{k}\fG\times^{\fQ}T$ is the gerbe of liftings of $T$ to $\fN$. If $T$ is trivial, it is immediate to check that $[\fX/\fG]\times^{\fQ}T\simeq [X/\fG]$, hence in general $[\fX/\fG]\times^{\fQ}T$ is a twisted form of $[X/\fG]$.

If $T$ does not lift to $\fN$, the coarse moduli space of $[\fX/\fG]\times^{\fQ}T$ is a twisted $G$-quotient whose associated gerbe has no sections, this proves the first part of Theorem~\ref{theorem:cohomology}. The second part is a direct consequence of the following Proposition~\ref{prop:cohomology}.

\begin{proposition}\label{prop:cohomology}
	Assume that $\fN$ is the entire normalizer of $\fG$. Every twisted $G$-quotient is the coarse moduli space of a twist
	\[[\fX/\fG]\times^{\fQ}T\to \cB_{k}\fG\times^{\fQ}T\]
	for some $\fQ$-torsor $T\to\spec k$.
\end{proposition}

\begin{proof}
	Let $Y\dashrightarrow\cG$ be a twisted $G$-quotient over $k$. Define a category fibered in sets $T$ as follows. If $S$ is a scheme over $k$, $T(S)$ is the set of isomorphisms 
	\[\fX_{S}/\fG_{S}\xrightarrow{\sim}Y_{S}\]
	for which there exists an fppf covering $S'\to S$ and a $2$-commutative diagram
	\[\begin{tikzcd}
		\fX_{S'}/\fG_{S'}\dar\ar[r,dashed]	&	\cB_{S'}\fG_{S'}\dar		\\
		Y_{S'}\ar[r,dashed]	&	\cG_{S'}
	\end{tikzcd}\] 
	where the vertical arrows are isomorphisms. We stress that the fppf covering and the $2$-commutative diagram are \emph{not} part of the datum, we are only selecting the isomorphisms for which such a diagram exists. It is clear that $T$ is a subsheaf of the sheaf of isomorphisms between $\fX/\fG$ and $Y$. The action of $\fQ$ on $\fX/\fG$ induces an action on $T$, and Lemma~\ref{lemma:autocover} implies that $T$ is a $\fQ$-torsor. It is straightforward to check that the given $G$-quotient is the twist of $\fX/\fG\dashrightarrow\cB_{k}\fG$ by $T$.
	
\end{proof}

Sometimes, it is interesting to study whether the structure actually descends to $\fX$, as opposed to a twist of $\fX$. For instance, if $\fX=\PP^{1}_{k}$ and the structure is a divisor $D\subset \PP^{1}_{K}$, one may want to study whether $D$ descends to $\PP^{1}_{k}$ \cite{marinatto} \cite{giulio-divisor}, but Theorem~\ref{theorem:GstrGtwist} only tells us whether $D$ descends to a Brauer-Severi variety of dimension $1$. Similarly, one might want to study if the embedding of a curve in $\PP^{2}$ is defined over the field of moduli, and not only if the curve embeds in a Brauer-Severi surface over the field of moduli.

\begin{definition}
	Let $X,G$ be as above, and $\fX$ a model of $X$ over $k$. A $G$-structure $\cX\to\cG$ of $X$ with field of moduli $k$ is \emph{$\fX$-neutral} if there exists a $2$-cartesian diagram
	\[\begin{tikzcd}
		\fX\dar\rar			&	\cX\dar	\\
		\spec k\rar["p"]	&	\cG
	\end{tikzcd}\]
	such that the induced embedding $\underaut_{\cG}(p)\subset\underaut(\fX)$ identifies $\underaut_{\cG}(p)(K)$ with $G$.
\end{definition}

Consider the quotient sheaf $\underaut(\fX)/\fN$. Taking the fibers of $\underaut(\fX)\to \underaut(\fX)/\fN$ defines a function $(\underaut(\fX)/\fN)(k)\to\H^{1}(k,\fN)$. Notice that these are sets, not groups. Still, they have a preferred object and it makes sense to consider the kernel 
\[\rK=\ker(\H^{1}(k,\fN)\to\H^{1}(k,\underaut(\fX))),\]
which is the image of $(\underaut(\fX)/\fN)(k)$, i.e. we have a long exact sequence for non-abelian cohomology. The elements of $\rK$ are the $\fN$-torsor which define the trivial twist of $\fX$.

The following version of Theorem~\ref{theorem:cohomology} is, again, a direct consequence of Proposition~\ref{prop:cohomology}

\begin{theorem}
	 If the composition 
	 \[(\underaut(\fX)/\fN)(k)\twoheadrightarrow\rK\hookrightarrow\H^{1}(k,\fN)\to\H^{1}(k,\fQ)\]
	 is not surjective, there exists a non-$\fX$-neutral $G$-structure with field of moduli equal to $k$. If $\fN$ is the entire normalizer of $\fG$ in $\underaut(\fX)$ the converse holds.
\end{theorem}

\section{Distinguished subsets}\label{sec:dist}

As we have seen above, $G$-structures with field of moduli $k$ correspond to twisted $G$-quotients of $X$ over $k$. If $Y$ is a twisted $G$-quotient over $k$ there is a rational map $Y\dashrightarrow\cG$ where $\cG$ is the residual gerbe and the structure is defined over the field of moduli if and only if $\cG(k)\neq\emptyset$.

If $y\in Y(k)$ is a smooth rational point and $G$ has degree prime with $\cha k$, then by the Lang-Nishimura theorem for tame stacks \cite[Theorem 4.1]{giulio-angelo-valuative} we have that $\cG(k)\neq\emptyset$ and hence the $G$-structure is defined over $k$. The smoothness assumption on $y$ can be relaxed, see \cite[\S 6]{giulio-angelo-moduli}, \cite{giulio-tqs2}. Because of this, it is important to have a framework to construct rational points (or, more generally, subspaces) of arbitrary twisted $G$-quotients. The proof of the following is straightforward.

\begin{lemma}\label{lemma:descent}
	A closed subscheme $Z\subset X$ descends to a twisted $G$-quotient $Y$ if and only if $\phi(Z)=Z$ for every $\phi\in\cN_{Y}$.\qed
\end{lemma}

Write $\cN_{X/k,G}\subset\aut_{k}(X)$ for the normalizer of $G$ in $\aut_{k}(X)$.

\begin{lemma}\label{lemma:distinctive}
	We have $\cN_{Y}\subset\cN_{X/k,G}$.
\end{lemma}

\begin{proof}
	If $\phi\in\cN_{Y}$ and $g\in G$, then clearly $\phi^{-1}\circ g\circ\phi\in\cN_{Y}\cap\aut_{K}(X)=G$, so $\phi\in\cN_{X/k,G}$.
\end{proof}

\begin{definition}
	A closed subset $Z\subset X$ is a \emph{distinguished subset} if, for every $\tau\in\cN_{X/k,G}$, $\tau(Z)=Z$.
\end{definition}

A distinguished subset is $G$-invariant since $G\subset\cN_{X/k,G}$, but the converse is false. Write $\pi$ for the projection $X\to X/G$.

\begin{lemma}\label{lemma:dist}
	Let $Z\subset X$ be a distinguished subset and $Y$ a twisted $G$-quotient over $k$. Then $\pi(Z)\subset X/G$ descends to a closed subset of $Y$.
\end{lemma}

\begin{proof}
	Follows directly from Lemma~\ref{lemma:descent} and Lemma~\ref{lemma:distinctive}.
\end{proof}

Conjugation defines a homomorphism $\cN_{X/k,G}\to\aut(G)$, let $\cA_{X/k,G}\subset\aut(G)$ be its image. Clearly, $\cA_{X/k,G}$ contains every inner automorphism of $G$.

While distinguished subsets are naturally defined in terms of the group $\cN_{X/k,G}\subset\aut_{k}(X)$, it is often sufficient to have knowledge about $\cA_{X/k,G}\subset\aut(G)$. Let us give some examples of this.

\begin{example}\label{example:dist}
	If $H\subset G$ is a subgroup and $Z\subset X$ a subspace, we say that $H$ stabilizes (resp. fixes) $Z$ if the elements of $H$ restrict to automorphisms (resp. to the identity) on $Z$. 
	
	Let $H\subset G$ be a subgroup, and let $S_{H}$ be the set of subgroups of $G$ of the form $\phi(H)$ for some $\phi\in\cA_{X/k,G}$. The following are distinguished subsets.
	\begin{itemize}
		\item The union and the intersection of the fixed loci of elements of $S_{H}$.
		\item Let $n$ be an integer, and suppose that $H$ stabilizes (resp. fixes) a finite number of irreducible closed subsets of dimension $n$, write $C_{n,H}$ for their union and similarly $C_{n,H'}$ for $H'\in S_{H}$. Then $\bigcup_{H'\in S_{H}}C_{n,H'}$ and $\bigcap_{H'\in S_{H}}C_{n,H'}$ are distinguished subsets.
		\item Suppose that $X$ is smooth and proper and that it has a line bundle $L$ whose class in the Néron-Severi group is $\aut_{k}(X)$-invariant, e.g. $X=\PP^{n}$ and $L=\cO(1)$. Then we can repeat the previous point but restricted to irreducible closed subset of a fixed degree $d$.
	\end{itemize}
\end{example}

Because of Example~\ref{example:dist} and the many other similar examples that can be given, we want to understand the group $\cA_{X/k,G}\subset\aut(G)$. We have inclusions
\[\operatorname{Inn}(G)\subset \cA_{X/K,G}\subset\cA_{X/k,G}\subset\aut(G).\]
Usually, the subgroup $\cA_{X/K,G}$ is easy to understand, since it is defined in terms of ``geometric'' $K$-automorphisms of $X$ and it is normal in $\cA_{X/k,G}$. It remains to study its cokernel. 

Suppose that $X,G$ descend to $\fX,\fG$ over $k$ with $\fG$ a group scheme, and that the action of $G$ on $X$ descends to an action of $\fG$ on $\fX$. In particular, we have an action of $\gal(K/k)$ on $X$ and hence a section $\gal(K/k)\to\aut_{k}(X)$ of $\aut_{k}(X)\to\gal(K/k)$. Furthermore, we have an action $\gal(K/k)\to\aut(G)$, $\sigma\mapsto\phi_{\sigma}$ on $G=\fG(K)$.

\begin{proposition}
	The image of $\gal(K/k)\to\aut_{k}(X)$ is contained in $\cN_{X/k,G}$, and $\cN_{X/k,G}$ is generated by $\cN_{X/K,G}$ and $\gal(K/k)$.
\end{proposition}

\begin{proof}
	Choose an element $\sigma\in\gal(K/k)$, let $\phi_{\sigma}$ be the induced automorphism of $G=\fG(K)$. Since the action $\rho:G\times X\to X$ descends to an action $\fG\times \fX\to \fX$ over $k$, we have a commutative diagram
	\[\begin{tikzcd}
		G\times X\ar[r,"\rho"]\ar[d,"\phi_{\sigma}\times\sigma^{*}"]	&	X\ar[d,"\sigma^{*}"]	\\
		G\times X\ar[r,"\rho"]											&	X
	\end{tikzcd}\]
	where the vertical arrows are $\sigma$-equivariant. It follows that $\sigma^{*}:X\to X$ is $\phi_{\sigma}$-equivariant, and hence $\sigma^{*}\in\cN_{X/k,G}$.
	
	Now consider an element $\tau\in\cN_{X/k,G}\subset\aut_{k}(X)$ normalizing $G$, and let $\sigma$ be its image in $\gal(K/k)$. Since $\sigma^{*}\in\cN_{X/k,G}$, then $\tau\circ\sigma^{*-1}$ is an element of $\cN_{X/k,G}$ which maps to the identity in $\gal(K/k)$, i.e. $\tau\circ\sigma^{*-1}\in\cN_{X/K,G}$.
\end{proof}

\begin{corollary}
	The subgroup $\cA_{X/k,G}\subset\aut(G)$ is generated by $\cA_{X/K,G}$ and by the image of $\gal(K/k)\to\aut(G)$. In particular, if $\fG$ is an inner form of $G$ then $\cA_{X/k,G}=\cA_{X/K,G}$.
\end{corollary}

\bibliographystyle{amsalpha}
\bibliography{fmod}

\end{document}